\documentclass{amsart}

\usepackage[T1]{fontenc}
\usepackage[utf8]{inputenc}
\usepackage[a4paper,margin=1.1in]{geometry}
\usepackage{amsmath,amssymb,amsfonts,mathtools}
\usepackage{enumitem}
\usepackage{microtype}
\usepackage{comment}
\usepackage{mathrsfs}
\usepackage[colorlinks=true,linkcolor=blue,citecolor=blue,urlcolor=blue]{hyperref}

\numberwithin{equation}{section}

\newtheorem{theorem}{Theorem}[section]

\newtheorem{proposition}[theorem]{Proposition}
\newtheorem{lemma}[theorem]{Lemma}

\newtheorem{remark}[theorem]{Remark}

\newcommand{\R}{\mathbb R}

\newcommand{\cN}{\mathcal N}

\newcommand{\cF}{\mathcal F}

\newcommand{\omegainf}{\omega_\infty}

\newcommand{\ip}[2]{\left\langle #1,#2\right\rangle}
\newcommand{\norm}[1]{\left\|#1\right\|}

\title[]
{A new proof of uniqueness of ground state for pseudo-relativistic Hartree equation}
\author{Pan Chen, Xiaoyu Zeng}

\begin{document}
\maketitle

\noindent{ \bf Abstract.}
   We present a new proof of the uniqueness of ground
states for the pseudo-relativistic Hartree equation in the non-relativistic
regime. The proof treats the Euler-Lagrange equation and the mass constraint
as a single augmented system, with the Lagrange multiplier as an additional
unknown. The key point is the invertibility of the augmented linearization at
the non-relativistic ground state. 


\section{Introduction}

In this paper, we study standing waves solutions of the pseudo-relativistic Hartree equation
\begin{equation}\label{1.0}
    i\partial_t \psi
    =
    \big(\sqrt{-\Delta+m^2}-m\big)\psi
    -
    \big(|x|^{-1}*|\psi|^2\big)\psi,
    \qquad (t,x)\in \mathbb{R}\times\mathbb{R}^3,
\end{equation}
where $m>0$ is the rest mass.  The operator
$\sqrt{-\Delta+m^2}-m$ is defined by the Fourier multiplier
$\sqrt{|\xi|^2+m^2}-m$ and describes the relativistic kinetic energy,
whereas the Hartree term with kernel $|x|^{-1}$ represents Newtonian
attraction.

Standing waves of \eqref{1.0} are of the form
\[
    \psi(t,x)=e^{i\mu t}u(x),
\]
and $u$ solves
\begin{equation}\label{1.01}
    \big(\sqrt{-\Delta+m^2}-m\big)u +\mu u
    -
    \big(|x|^{-1}*|u|^2\big)u
    =0
    .
\end{equation}
The existence of solutions to \eqref{1.01} can be addressed by minimizing
 the energy functional
\begin{equation*}
    \mathcal{E}(u)
    :=
    \int_{\mathbb{R}^3}
        {u}\big(\sqrt{-\Delta+m^2}-m\big)u\,dx
    -
    \frac12
    \int_{\mathbb{R}^3}
        \big(|x|^{-1}*|u|^2\big)|u|^2\,dx
\end{equation*}
under the mass constraint $\|u\|_{L^2}^2=N$.  Equivalently, one studies
\begin{equation}\label{noneq}
    e(N)
    :=
    \inf\left\{
        \mathcal{E}(u):
        u\in H^{1/2}(\mathbb{R}^3),\ \|u\|_{L^2}^2=N
    \right\}.
\end{equation}
The parameter $N$ is the stellar mass of the corresponding boson star.
There exists a critical mass $N_*>0$ which corresponds to
 the Chandrasekhar limit mass, such
that minimizers exist exactly in the subcritical 
range $0<N<N_*$, see \cite{LiebYau1987} for more details.

The model has its origins in the work of Lieb and Thirring on
relativistic gravitational collapse \cite{LiebThirring1984} and in the
work of Lieb and Yau on the Chandrasekhar theory and the Hartree
approximation for bosonic stars \cite{LiebYau1987}.  A rigorous
mean-field derivation of the time-dependent boson-star equation was
obtained in \cite{ElgartSchlein2007}.  The analysis
of pseudo-relativistic Hartree equations has been developed in
many directions; see, for example,
\cite{ChoOzawa2006,Lenzmann2007,Lenzmann2009,FroehlichJonssonLenzmann2007,
FroehlichJonssonLenzmann2007b,FroehlichLenzmann2007,GuoZeng2020,
ChoiHongSeok2018,Chen,J. Bellazzini}.

An open problem in this theory is the Lieb--Yau uniqueness
conjecture: for every $0<N<N_*$, the variational problem \eqref{noneq}
should have a unique positive minimizer, up to phase and translation
\cite{LiebYau1987}.  This problem is delicate for two reasons.
First, both the kinetic operator and the Hartree interaction are
nonlocal, so the ODE methods available for local nonlinear
Schr\"odinger equations do not apply directly.  Second, the presence of
the rest mass $m>0$ breaks the simple scaling structure that would
otherwise connect ground states of different $L^2$-norm.

The first major progress was due to Lenzmann \cite{Lenzmann2009}, who
proved uniqueness in the small-mass regime $0<N\ll N_*$, apart from at
most countably many exceptional masses.  His argument combines the
implicit function theorem, the non-relativistic limit, and the
nondegeneracy of the ground state of the limiting Hartree equation.  The
exceptional set is tied to the differentiability of the ground-state
energy $e(N)$, which is used to identify the Lagrange multiplier as a
function of the mass.  Guo and Zeng later removed this exceptional set
and proved the small-mass case of the Lieb-Yau conjecture for all
sufficiently small $N>0$ \cite{GuoZeng2020}, using a Pohozaev-type
identity.  The full conjecture on the whole interval $0<N<N_*$ remains
open.

In order to formulate the small-mass problem in a normalized form, 
motivated by \cite{Lenzmann2009}, for $c>0$, we define the functional
\begin{equation*}
    \mathcal{E}_c(u)
    :=
    \left\langle
        \big(\sqrt{-c^2\Delta+m^2c^4}-mc^2\big)u,u
    \right\rangle
    -
    \frac12
    \int_{\mathbb{R}^3}
        \big(|x|^{-1}*u^2\big)u^2\,dx .
\end{equation*}
If $u$ is a minimizer of \eqref{noneq} with $N=c^{-1}$, then
$\widetilde u(x)=c^2u(cx)$ is a minimizer of
\begin{equation}\label{1.1}
    e_c
    :=
    \inf\left\{
        \mathcal{E}_c(u):
        u\in H^{1/2}(\mathbb{R}^3),\ \|u\|_{L^2}^2=1
    \right\}.
\end{equation}
Thus the small-mass limit $N\to0$ is equivalently described by the
non-relativistic limit $c\to\infty$ in the normalized problem
\eqref{1.1}.  Writing
\[
    P_c(D):=\sqrt{-c^2\Delta+m^2c^4}-mc^2,
    \qquad
    P_\infty(D):=-\frac{1}{2m}\Delta,
\]
we have $P_c(D)\to P_\infty(D)$ as $c\to\infty$, and the limiting
variational problem is
\begin{equation}\label{1.2}
    e_\infty
    =
    \inf_{\|u\|_{L^2}=1}
    \mathcal{E}_\infty(u),
\end{equation}
where
\[
    \mathcal{E}_\infty(u)
    =
    \left\langle P_\infty(D)u,u\right\rangle
    -
    \frac12
    \int_{\mathbb{R}^3}
        \big(|x|^{-1}*u^2\big)u^2\,dx .
\]
The Euler-Lagrange equations associated with \eqref{1.1} and
\eqref{1.2} are, respectively,
\begin{equation}\label{1.3}
    \begin{cases}
        P_c(D)u_c+\omega_c u_c=
        (|x|^{-1}*|u_c|^2)u_c,\\
        \|u_c\|_{L^2}=1,
    \end{cases}
\end{equation}
and
\begin{equation}\label{1.4}
    \begin{cases}
        P_\infty(D)u_\infty+\omega_\infty u_\infty=
        (|x|^{-1}*|u_\infty|^2)u_\infty,\\
        \|u_\infty\|_{L^2}=1,
    \end{cases}
\end{equation}
where $\omega_c$ and $\omega_\infty$ are the corresponding Lagrange
multipliers.  The limiting equation \eqref{1.4} is the classical
non-relativistic Hartree equation, also known as the Schr\"odinger-Newton
equation.  Existence and uniqueness of its ground state were proved in
\cite{Lieb1977,TodMoroz1999}, and its nondegeneracy was established in
\cite{Lenzmann2009,Wei}.  This nondegeneracy is also a fundamental property
in the study of solitary-wave dynamics; see
\cite{Lenzmann2009,FroehlichTsaiYau2002,
FroehlichGustafsonJonssonSigal2004,JonssonFroehlichGustafsonSigal2006}.

The purpose of this paper is to give a shorter and more structural proof
of the small-mass uniqueness theorem, or equivalently the uniqueness of
\eqref{1.1} for all sufficiently large $c$. Our main result can be stated as:
\begin{theorem}\label{main}
If $c > 0$ is large enough, then variational problem \eqref{1.1} admits
a unique positive minimizer, up to phase and translation.
\end{theorem}

The main point of our new proof is to treat the Euler-Lagrange equation and
the mass constraint as one augmented system.  More precisely, we regard
the Lagrange multiplier as an unknown parameter and solve
\[
    \mathcal F(\beta,u,\omega)
    =
    \left(
        u-R_{\beta,\omega}\mathcal N(u),
        \|u\|_{L^2}^2-1
    \right)
    =0,
    \qquad
    \beta=c^{-1},
\]
near the limiting solution $(0,u_\infty,\omega_\infty)$, where
\[
    \mathcal N(u)=\big(|x|^{-1}*u^2\big)u,
    \qquad
    R_{\beta,\omega}
    =
    \begin{cases}
        (P_\infty(D)+\omega)^{-1}, & \beta=0,\\
        (P_{1/\beta}(D)+\omega)^{-1}, & \beta>0.
    \end{cases}
\]
With this formulation the normalization condition is included in the
nonlinear map itself.  The linearization with respect to $(u,\omega)$ at
$(0,u_\infty,\omega_\infty)$ is
\[
    \partial_{(u,\omega)}\mathcal F(0,u_\infty,\omega_\infty)(\phi,\tau)
    =
    \left(
        R_{0,\omega_\infty}(\mathcal{L}\phi+\tau u_\infty),
        2\left\langle u_\infty,\phi\right\rangle
    \right),
\]
where \(\mathcal L\) is the linearization of the nonlinear Hartree equation \eqref{1.4} around the ground state
 \(u_\infty\), namely,
\[
    \mathcal{L}\phi
    :=
    P_\infty(D)\phi+\omega_\infty\phi-
    \mathcal N'(u_\infty)\phi .
\]
We prove that this augmented linearized operator is invertible by showing
\[
    \left\langle u_\infty,\mathcal L^{-1}u_\infty\right\rangle
    =
    -\frac{1}{4\omega_\infty}\neq0,
\]
 An
implicit-function-type argument then gives local uniqueness of solutions
to the augmented system near the non-relativistic ground state.

This approach differs from Lenzmann's argument in that the Lagrange
multiplier is not recovered from differentiability properties of the
energy curve $e(N)$, consequently no exceptional set of masses is
introduced.  It also differs from the method of Guo and Zeng, since we
do not need to analyze the limiting equation for the
$L^\infty$-normalized difference of two minimizers as $c\to\infty$.
Instead, the proof reduces the uniqueness statement to the invertibility
of the augmented linearization, proved below in Proposition
\ref{invert2}.

\textbf{Notations. }
Throughout this paper, we make use of the following notations.
\begin{itemize}
    \item $\|\cdot\|_{L^q}$ denotes the usual norm of the Lebesgue space $L^q\left(\mathbb{R}^3\right)$;
    \item $\|\cdot\|_{H^s}$ denotes the usual norm of the Sobolev space $H^s\left(\mathbb{R}^3\right)$;
	\item $\ip{\cdot}{\cdot}$ denotes the usual scalar product of 
	the Hilbert space $L^2\left(\mathbb{R}^3\right)$;
    \item $C$ is some positive constant that may change from line to line;
    \item \(\|P\|_{L(H,K)}\) denotes the operator
     norm of the bounded linear operator \(P\) 
     from the Banach space \(H\) to \(K\). 
     When \(H=K\), it is denoted 
     by \(\|P\|_{L(H)}\).
\end{itemize}

\section{Proof of Main Theorem}

For the
reader's convenience, we provide a concise 
summary of the properties of ground states to
 \eqref{1.3} and \eqref{1.4}, for more details, we refer reader to \cite{Lenzmann2009,ChoiHongSeok2018,Chen,GuoZeng2020}

\begin{proposition}
	\label{limit}
Let \(u_c\) (resp. \(u_\infty\)) be any ground state 
of the variational problem \eqref{1.1} (resp. \eqref{1.2}).
 Then, up to phase and translation, 
 both functions are positive and symmetric-decreasing. 
 Each satisfies its corresponding equation, 
 \eqref{1.3} or \eqref{1.4}, with a Lagrange multiplier
  \(\omega_c>0\) or \(\omega_\infty>0\), respectively.
   For any \(s>0\), \(u_c\) and \(u_\infty\) belong 
   to \(H^s(\mathbb{R}^3)\). Moreover, as \(c\to\infty\),
 \[
 u_c \to u_\infty \quad \text{in } H^s(\mathbb{R}^3),\qquad
 \omega_c \to \omega_\infty,\qquad e_c \to e_\infty.
 \]
\end{proposition}
{
We recall that the unique positive radial ground state
$u_\infty$ to \eqref{1.4} is nondegenerate in $H^1_r(\mathbb{R}^3)$
(see \cite[Proposition 2]{Lenzmann2009} and \cite[Theorem III.1]{Wei}),
 that means the linearized equation 
 $$
 \mathcal{L}u=0
 $$
 has only the trivial solution $u=0$ in $H^1_r(\mathbb{R}^3)$.
 The precise knowledge implies  $\mathcal{L} : H^{s+2}_r(\mathbb{R}^3)\to H^{s}_r(\mathbb{R}^3)$ 
 is invertible (see \cite{Chen}).
}
\begin{lemma}\label{invert}
    For each $s\geq 0$, the linearized operator  
    $\mathcal{L} : H^{s+2}_r(\mathbb{R}^3)\to H^{s}_r(\mathbb{R}^3)$
    is
invertible.
\end{lemma}

Let's recall resolvent operator
\begin{equation}\label{beta}
   R_{\beta,\omega}=
   \begin{cases}
   (P_\infty(D)+\omega)^{-1},& \beta=0,\\[1mm]
   (P_{1/\beta}(D)+\omega)^{-1},& \beta>0.
   \end{cases}
\end{equation}
Using the Fourier transform, we obtain the following results.
\begin{lemma}\label{continuity}
Let $I\subset\subset(0,\infty)$ be compact. If $\beta\to0$ and $\omega\to\omega_0\in I$, then
\begin{equation}\label{c1}
   \norm{R_{\beta,\omega}-R_{0,\omega_0}}_{ L(H^s)}\to0
\end{equation}
and
\begin{equation}\label{c2}
   \norm{R_{\beta,\omega}^2-R_{0,\omega_0}^2}_{ L(H^s)}\to0
\end{equation}
for every $s\ge0$. 
\end{lemma}
{
\begin{proof}
   For \(\beta>0\), the Fourier symbol of
\(P_{1/\beta}(D)\) is
\[
p_\beta(\xi)
=
\sqrt{\beta^{-2}|\xi|^2+m^2\beta^{-4}}-m\beta^{-2}
=
\frac{|\xi|^2}{\sqrt{m^2+\beta^2|\xi|^2}+m}.
\]
For \(\beta=0\), the limiting symbol is
\[
p_0(\xi)=\frac{|\xi|^2}{2m}.
\]
Thus
\[
\widehat{R_{\beta,\omega}f}(\xi)
=
r_{\beta,\omega}(\xi)\widehat f(\xi),
\qquad
r_{\beta,\omega}(\xi):=\frac{1}{p_\beta(\xi)+\omega}.
\]
Then, we have
\[
\|(R_{\beta,\omega}-R_{0,\omega_0})f\|_{H^s}
\le
\sup_{\xi\in\mathbb R^3}
|r_{\beta,\omega}(\xi)-r_{0,\omega_0}(\xi)|
\,\|f\|_{H^s}.
\]
It is therefore enough to prove
\[
\sup_{\xi\in\mathbb R^3}
|r_{\beta,\omega}(\xi)-r_{0,\omega_0}(\xi)|
\to0.
\]
Choose \(a>0\) such that
\[
\omega\ge a,\qquad \omega_0\ge a.
\]
 Put \(t=|\xi|^2\), for \(0\le \beta\le1\),
\[
p_\beta(\xi)
=
\frac{t}{\sqrt{m^2+\beta^2t}+m}
\ge
\frac{t}{\sqrt{m^2+t}+m}
=:q(t),
\]
and \(q(t)\to\infty\) as \(t\to\infty\). Hence, if \(t\ge T\), then
\[
|r_{\beta,\omega}(\xi)-r_{0,\omega_0}(\xi)|
\le
\frac{1}{p_\beta(\xi)+\omega}
+
\frac{1}{p_0(\xi)+\omega_0}
\le
\frac{2}{q(T)+a}.
\]
Given \(\varepsilon>0\), choose \(T\gg1\), such that
\[
\frac{2}{q(T)+a}<\frac{\varepsilon}{2}.
\]
On  compact region \(|\xi|^2\le T\),
\[
p_\beta(\xi)
=
\frac{|\xi|^2}{\sqrt{m^2+\beta^2|\xi|^2}+m}
\to
\frac{|\xi|^2}{2m}
=
p_0(\xi)
\]
uniformly as \(\beta\to0\). Therefore,
\[
\sup_{|\xi|^2\le T}
|p_\beta(\xi)-p_0(\xi)|\to0.
\]
Thus, for \(\beta\) sufficiently small and \(\omega\) sufficiently close to \(\omega_0\),
\[
\begin{aligned}
\sup_{\xi\in\mathbb R^3}|r_{\beta,\omega}(\xi)-r_{0,\omega_0}(\xi)|
&=\sup_{\xi\in\mathbb R^3}
\left|
\frac{1}{p_\beta(\xi)+\omega}
-
\frac{1}{p_0(\xi)+\omega_0}
\right|        \\
&\le \varepsilon,
\end{aligned}
\]
which proves \eqref{c1}. Similarly, we can obtain \eqref{c2}.
\end{proof}

The following trilinear estimate for the Hartree nonlinearity follows from \cite{ChoiHongSeok2018}.

\begin{lemma}\label{trilinear}
For $s\geq \frac{1}{2}$, $u,v,w\in H^s(\mathbb{R}^3)$,
 there exists \(C>0\) such that
\[
\|(|x|^{-1}*uv)w\|_{H^s}
\le
C\|u\|_{H^s}\|v\|_{H^s}\|w\|_{H^s}.
\]
\end{lemma}

Let's recall the augmented system
\begin{equation*}
   \cF(\beta,u,\omega)
   =\left(u-R_{\beta,\omega}\cN(u),\ \|u\|_{L^2}^2-1\right),
\end{equation*}
where
\[
   \cF:[0,\beta_0)\times H^2_r(\mathbb{R}^3)\times I\to  H^2_r(\mathbb{R}^3)\times\R
\]
for a compact interval $I\subset\subset(0,\infty)$ containing $\omega_\infty$. For $\beta>0$, the equation $\cF(\beta,u,\omega)=0$ is equivalent to
\begin{equation}\label{eq:eq-beta-positive}
   P_{1/\beta}u+\omega u=\cN(u),
   \qquad \|u\|_{L^2}=1.
\end{equation}
For $\beta=0$, $\omega= \omegainf$, it is the limiting equation \eqref{1.4}.
Direct computation yields the Gâteaux derivative of $\mathcal{F}(\beta, u,\omega)$ with respect to $(u, \omega)$ is
\[
\partial_{(u,\omega)}\mathcal F(\beta,u,\omega)(\phi,\tau)
=
\left(
\phi
-
R_{\beta,\omega}\mathcal N'(u)\phi
+
\tau R_{\beta,\omega}^2\mathcal N(u),
\,
2\langle u,\phi\rangle
\right).
\]
It follows from Lemma \ref{trilinear},
\[
\|\mathcal N(u_1)-\mathcal N(u_2)\|_{H^2}
\le
C\bigl(\|u_1\|_{H^2}^2+\|u_2\|_{H^2}^2\bigr)
\|u_1-u_2\|_{H^2},
\]
and
\[
\|\mathcal N'(u_1)-\mathcal N'(u_2)\|_{L(H^2)}
\le
C\bigl(\|u_1\|_{H^2}+\|u_2\|_{H^2}\bigr)
\|u_1-u_2\|_{H^2}.
\]
Moreover, by Lemma \ref{continuity},
\[
R_{\beta,\omega}
\to
R_{0,\omega_\infty}, \quad
R_{\beta,\omega}^2
\to
R_{0,\omega_\infty}^2
\quad
\text{in }\quad  L(H^2)
\]
as
\[
\beta\to0,
\qquad
\omega\to\omega_\infty.
\]
Therefore
\[
\left\|\partial_{(u,\omega)}\mathcal{F}(\beta,u,\omega)
-
\partial_{(u,\omega)}\mathcal{F}(0,u_\infty,\omega_\infty)\right\|_{ L\left(
H^2_r(\mathbb R^3)\times\mathbb R
\right)}\to 0
\]
as
\[
\beta\to0,\quad \omega\to\omega_\infty,\quad
u\to u_\infty\,\, \text{ in } \,\,H^2.
\]
Consequently, since the Gâteaux derivative \(\partial_{(u,\omega)}\mathcal{F}(\beta,u,\omega)\) is continuous at \((0, u_\infty, \omega_\infty)\), it follows that the Fréchet derivative of \(\mathcal{F}(\beta, u,\omega)\) at \((0, u_\infty, \omega_\infty)\) with respect to \((u, \omega)\) exists and is given by
\[
\partial_{(u,\omega)}\mathcal F(0,u_\infty,\omega_\infty)(\phi,\tau)
=
\left(
\phi
-
R_{0,\omega_\infty}\mathcal N'(u_\infty)\phi
+
\tau R_{0,\omega_\infty}^2\mathcal N(u_\infty),
\,
2\langle u_\infty,\phi\rangle
\right).
\]
Since \(u_\infty\) satisfies the equation
\[
P_\infty(D)u_\infty+\omega_\infty u_\infty
=
\mathcal N(u_\infty),
\]
we have
\begin{equation}\label{DF}
  \partial_{(u,\omega)}\mathcal F(0,u_\infty,\omega_\infty)(\phi,\tau)
=
\left(
R_{0,\omega_\infty}(\mathcal L\phi+\tau u_\infty),
\,
2\langle u_\infty,\phi\rangle
\right). 
\end{equation}

}

\begin{proposition}\label{invert2}
The operator
\begin{equation*}
   \mathscr{L}:=\partial_{(u,\omega)}\cF(0,u_\infty,\omega_\infty):H_r^2(\mathbb{R}^3)\times\R\to H_r^2(\mathbb{R}^3)\times\R
\end{equation*}
is invertible.
\end{proposition}

\begin{proof}
It is enough to solve
\begin{equation*}
   \mathscr{L}(\phi,\tau)=(f,a),
   \qquad (f,a)\in H_r^2(\mathbb{R}^3)\times\R.
\end{equation*}
By \eqref{DF}, this is equivalent to
\begin{equation}\label{lineq}
   \mathcal{L} \phi+\tau u_\infty=(P_\infty(D)+\omega_\infty)f,
   \qquad
   2\langle u_\infty,\phi\rangle=a.
\end{equation}
By Lemma \ref{invert},
\begin{equation}\label{linereq}
   \phi=\mathcal{L} ^{-1}(P_\infty(D)+\omega_\infty)f-\tau \mathcal{L} ^{-1}u_\infty.
\end{equation}
The second equation in \eqref{lineq} determines $\tau$ uniquely provided
\begin{equation}\label{eq0}
   \ip{u_\infty}{\mathcal{L} ^{-1}u_\infty}\neq0.
\end{equation}
We now prove \eqref{eq0}. For $\lambda>0$ define
\begin{equation}\label{eq:scaling}
   u_\infty^\lambda(x):=\lambda^2u_\infty(\lambda x).
\end{equation}
then
\begin{equation}\label{eq}
   P_\infty(D) u_\infty^\lambda+\lambda^2\omega_\infty u_\infty^\lambda
   =\cN(u_\infty^\lambda).
\end{equation}
Differentiating \eqref{eq} at $\lambda=1$ yields
\begin{equation*}
   \mathcal{L}  v=-2\omega_\infty u_\infty,
   \qquad
   v:=x\cdot\nabla u_\infty+2u_\infty.
\end{equation*}
Consequently
\begin{equation}\label{2.15}
   \mathcal{L} ^{-1}u_\infty=-\frac{1}{2\omega_\infty}v.
\end{equation}
Since $\norm{u_\infty}_2=1$,
\begin{equation}\label{eq:pairingZ}
   \ip{u_\infty}{x\cdot\nabla u_\infty}
   =-\frac32\norm{u_\infty}_2^2=-\frac32,
\end{equation}
so
\begin{equation}\label{2.17}
   \ip{u_\infty}{v}=\frac12.
\end{equation}
Combining \eqref{2.15} and \eqref{2.17} gives
\begin{equation*}
   \ip{u_\infty}{\mathcal{L} ^{-1}u_\infty}
   =-\frac{1}{2\omega_\infty}\ip{u_\infty}{v}
   =-\frac{1}{4\omega_\infty}\neq0.
\end{equation*}
Therefore $\tau$ is uniquely determined by
\[
   2\ip{u_\infty}{\mathcal{L} ^{-1}(P_\infty+\omega_\infty)f}
   -2\tau\ip{u_\infty}{\mathcal{L} ^{-1}u_\infty}=a,
\]
and then $\phi$ is uniquely determined by \eqref{linereq}. This proves that $\mathscr{L}$ is bijective. 
By the open mapping theorem, the inverse operator \(\mathscr{L}^{-1}\) is continuous.
\end{proof}

\
\begin{proposition}\label{prop}
There exist $\beta_1>0$ and $\delta>0$ such that, for every $0\le\beta\le\beta_1$, the equation
\begin{equation}\label{F=0}
   \cF(\beta,u,\omega)=0
\end{equation}
has at most one solution $(u,\omega)\in H^2_r(\mathbb{R}^3)\times I$ provided
\begin{equation}\label{loc}
   \norm{u-u_\infty}_{H^2}+|\omega-\omega_\infty|<\delta.
\end{equation}
\end{proposition}

\begin{proof}
 By the continuity of $\partial_{(u,\omega)}\cF$ at $(0,u_\infty,\omega_\infty)$, there are $\beta_1>0$ and $\delta>0$ such that, whenever $0\le\beta\le\beta_1$ and \eqref{loc} holds,
\begin{equation*}
   \norm{\mathscr{L}^{-1}\left(\partial_{(u,\omega)}\cF(\beta,u,\omega)-\mathscr{L}\right)}_{ L(H^2_r\times\R)}
   \le \frac12.
\end{equation*}
Suppose that $(u_1,\omega_1)$ and $(u_2,\omega_2)$ are two solutions of \eqref{F=0} satisfying \eqref{loc}. Set
\[
   h:=u_1-u_2,
   \eta:=\omega_1-\omega_2,
\]
and 
\[
   (u_t,\omega_t):=(u_2,\omega_2)+t(h,\eta),
   \qquad 0\le t\le1.
\]
If $\delta$ is chosen smaller if necessary, by Lemma \ref{continuity}, we have
\begin{equation}\label{2.22}
   \norm{\mathscr{L}^{-1}\left(\partial_{(u,\omega)}\cF(\beta,u_t,\omega_t)-\mathscr{L}\right)}_{ L(H^2_r\times\R)}
   \le \frac12, \quad 0\leq t\leq 1.
\end{equation}
It follows from $\cF(\beta,u_1,\omega_1)=\cF(\beta,u_2,\omega_2)=0$ that
\begin{equation*}
   0=\int_0^1\partial_{(u,\omega)}\cF(\beta,u_t,\omega_t)(h,\eta)\,dt.
\end{equation*}
Thus
\begin{equation}\label{2.24}
   \mathscr{L}(h,\eta)
   =-\int_0^1\left(\partial_{(u,\omega)}\cF(\beta,u_t,\omega_t)-\mathscr{L}\right)(h,\eta)\,dt.
\end{equation}
Applying $\mathscr{L}^{-1}$ to \eqref{2.24}, and using \eqref{2.22}, we obtain
\[
   \norm{(h,\eta)}_{H^2\times\R}
   \le \frac12\norm{(h,\eta)}_{H^2\times\R}.
\]
Therefore $h=0$ and $\eta=0$, which proves the proposition.
\end{proof}
{
\begin{remark}
One can show that $\mathcal{F}$ is $C^\infty$ in
$(\beta,u,\omega)$ for $\beta>0$, the implicit function theorem implies
that the unique local branch
\[
F(\beta,u_\beta,\omega_\beta)=0
\]
is $C^\infty$ for $0<\beta<\beta_1$. Hence $c\to(u_c,\omega_c)$ is $C^\infty$ for $c\gg1$, and so
\[
e_c=\mathcal E_c(u_c)
\]
is $C^\infty$ for $c\gg1$.
Using the scaling relation $e(N)=N^3e_{1/N}$, we also obtain that
 $N\to e(N)$ is $C^\infty$ for $0<N\ll N_*$.
\end{remark}
}

\begin{proof}[\textbf{Proof of Theorem \ref{main}}]
By Proposition \ref{limit}, for sufficiently large $c>0$,
 any symmetric-decreasing ground state $u_c$ of \eqref{1.3}
  satisfies the assumptions of Proposition \ref{prop}.
   Hence, we conclude that the symmetric-decreasing ground
    state $u_c$ of the variational problem \eqref{1.1} is
     unique, provided $c$ is large enough.
\end{proof}

\noindent {Pan Chen\\
School of Mathematical Sciences,\\
 Shanghai Jiao Tong University, Shanghai 200240, P.R. China
\\
e-mail: chenpan2020@amss.ac.cn }
\medskip\\
\noindent {Xiaoyu Zeng\\
Center for Mathematical Sciences and Department of Mathematics,\\
University of Technology, Wuhan, 430070, P.R. China\\
e-mail: xyzeng@whut.edu.cn}

\end{document}